\documentclass[12pt, twoside, leqno]{article}
\usepackage{amsmath,amsthm}
\usepackage{amssymb}
\usepackage{enumitem}
\usepackage{graphicx}

\usepackage[T1]{fontenc}
\usepackage[mathscr]{euscript}
\newtheorem{theorem}{Theorem}[section]

\newtheorem{lemma}[theorem]{Lemma}
\newtheorem{proposition}[theorem]{Proposition}

\newtheorem{fact}[theorem]{Fact}

\theoremstyle{definition}

\newtheorem{remark}[theorem]{Remark}

\numberwithin{equation}{section}

\begin{document}

\baselineskip=17pt


\title{Number of triple points on complete intersection Calabi-Yau threefolds}

\author{Kacper Grzelakowski\\
Wydzia\l \, Matematyki i Informatyki\\
Uniwerstytet \L\'odzki\\
Banacha 22\\
90-238 \L \'odz, Poland\\
E-mail: kacper.grzelakowski@wmii.uni.lodz.pl}

\date{}

\maketitle

\renewcommand{\thefootnote}{}



\renewcommand{\thefootnote}{\arabic{footnote}}
\setcounter{footnote}{0}


\begin{abstract}
We discuss bounds for the number of ordinary triple points on complete intersection Calabi-Yau threefolds in projective spaces and for Calabi-Yau threefolds in weighted projective spaces. In particular, we show that in $\mathbb{P}^5$ the intersection of a quadric and a quartic cannot have more than $10$ ordinary triple points. We provide examples of complete intersection Calabi-Yau threefolds with multiple triple points. We obtain the exact bound for a sextic hypersurface in $\mathbb{P}[1:1:1:1:2]$, which is $10$. We also discuss Calabi-Yau threefolds that cannot admit triple points.
\end{abstract}

\section{Introduction}

By an \emph{ordinary triple point} on a threefold $X$ we mean an isolated singularity $O$ such that the tangent cone to $X$ at $O$ is a cone over a smooth degree three surface. In this paper we consider bounds on the number of ordinary triple points on Calabi-Yau threefolds. Our interest was sparked off by the results of Kloosterman and Rams regarding the possible number of triple points on a Calabi-Yau quintic threefold in $\mathbb{P}^4$ \cite{KR}. They have proven that a quintic threefold with a reducible hyperplane section cannot have more than $10$ ordinary triple points as the only singularities and constructed an example where this bound is reached. The question whether or not a quintic threefold can have $11$ ordinary triple points remains open, but it is known that the number $12$ is impossible.

Our interest in Calabi-Yau threefolds admitting ordinary triple points as the only singularities comes from the fact that the resolution of such singularities is canonical - it does not modify the canonical class of a variety and leads to new smooth Calabi-Yau threefolds. For double points we can obtain such resolution only if there exists a so-called small resolution of singularities, that is, replacing a point by a $\mathbb{P}^1$. There already exists an extensive literature on the number of double points on Calabi-Yau threefolds, for example (\cite{GP}\cite{GK}\cite{DVS}), yet the exact bounds are still difficult to obtain. For instance, in the case of a quintic threefold the best result is given by van Straten in \cite{DVS} with $130$ double points and it is not known whether or not it is indeed the least upper bound. The case of ordinary triple points should be simpler, yet it is much less discussed.

The threefolds we consider are a natural next step after a quintic hypersurface in $\mathbb{P}^4$. We analyse Calabi-Yau threefolds that are complete intersection manifolds in projective spaces, namely $X_{2,4}, X_{3,3}\subset\mathbb{P}^5$, $X_{2,2,3}\subset \mathbb{P}^6$ and $X_{2,2,2,2}\subset\mathbb{P}^7$ where the subscripts are the degrees of the hypersurfaces we intersect. We also provide a bound for the number of ordinary triple points for a sextic in weighted projective space. Some of the results are quite straightforward but we have not been able to locate them in the existing literature. Numerical bounds are presented in Table 1. We include there the results of Rams and Kloosterman regarding the quintic.

We always assume that the threefold in question has only ordinary triple points (OTPs) as singularities. Let $\mu_3(X)$ denote the maximal number of ordinary triple points that $X$ can have under this condition. We drop $X$ if the context is clear. The analysis of $\mu_3(X_{2,4})$ reduces to the analysis of a quintic threefold as we perform a projection from one of the triple points of $X_{2,4}$ to obtain a quintic hypersurface in $\mathbb{P}^4$. In doing so we obtain a particular quintic, namely one containing a cubic surface and also admitting $24$ double points. Because of those additional singularities we are not able to use all the methods of \cite{KR} to analyse this quintic, yet we still make use of the Varchenko bound to find an upper bound on $\mu_3(X_{2,4})$ which is $10$. We have only been able to construct $X_{2,4}$ admitting $7$ ordinary triple points as singularities and so the problem to determine the exact bound remains open. 

The case of $X_{3,3}$ proves to be more complex. In our discussion we assume that all triple points of $X_{3,3}$ are inherited from some cubic fourfold, that is, for each $O$ triple on $X_{3,3}$ there exists a cubic fourfold $X_O$ such that $X_{3,3}\subset X_O$ and $O$ is triple on $X_O$. In that case we find the exact bound $\mu_3$ for a complete intersection threefold ${X_{3,3}}$, which is $9$, and we provide a construction of a $X_{3,3}$ admitting that many. The points lie in an interesting configuration as they come in three collinear triples. We discuss the geometry of a resolution of singularities of a threefold $X_{3,3}$ containing only inherited triple points, in particular the fact that it can be realised as a double cover of a cubic threefold.

We include one example of $X_{2,2,3}$ admitting multiple ordinary triple points. 

The next part of this paper is devoted to finding $\mu_3$ of a sextic hypersurface $X_6\subset\mathbb{P}(1:1:1:1:2)$. To find it we have to look at its hyperplane section, as it turns out that all the triple points are contained in one. We use the fact that the weighted projective space $\mathbb{P}(1:1:1:1:2)$ can be realized as a Veronese embedding of $\mathbb{P}^3$ in $\mathbb{P}^9$ and $X_6$ as its trisection and thus a triple cover of $\mathbb{P}^3$. We also employ the techniques used in finding the polar bound for the possible number of singular points. 

The paper ends with a discussion of Calabi-Yau threefolds of degree $8$ and $10$ in weighted projective spaces and of the complete intersection of four quadrics in $\mathbb{P}^7$ which do not admit any ordinary triple points.

	\begin{table}[] 
\centering 
		\caption{Bound on the number of triple points on Calabi-Yau threefolds}
\begin{tabular}{rc}
			\cline{1-2}
			\multicolumn{1}{|r|}{Calabi-Yau threefold $X$}                                      & \multicolumn{1}{c|}{$\mu_3(X)$} \\ \hline
			\multicolumn{1}{|r|}{$X_{5}\subset\mathbb{P}^4$}         & \multicolumn{1}{c|}{$10$ or $11$}        \\ \hline
			\multicolumn{1}{|r|}{$X_{2,4}\subset\mathbb{P}^5$}         & \multicolumn{1}{c|}{$7\leq\dots\leq 10$}       \\ \hline
			\multicolumn{1}{|r|}{$X_{3,3}\subset\mathbb{P}^5$}         & \multicolumn{1}{c|}{$ 9\leq$}       \\ \hline
			\multicolumn{1}{|r|}{$X_{2,2,3}\subset\mathbb{P}^6$}         & \multicolumn{1}{c|}{$4 \leq$}       \\ \hline
			\multicolumn{1}{|r|}{$X_{2,2,2,2}\subset\mathbb{P}^7$}         & \multicolumn{1}{c|}{0}       \\ \hline
			\multicolumn{1}{|r|}{$X_6\subset\mathbb{P}(1:1:1:1:2)$}    & \multicolumn{1}{c|}{10}      \\ \hline
			\multicolumn{1}{|r|}{$X_{8}\subset\mathbb{P}(1:1:1:1:4)$}  & \multicolumn{1}{c|}{0}       \\ \hline
			\multicolumn{1}{|r|}{$X_{10}\subset\mathbb{P}(1:1:1:2:5)$} & \multicolumn{1}{c|}{0}       \\ \hline
			\multicolumn{1}{l}{}                                       & \multicolumn{1}{l}{}         \\
			\multicolumn{1}{l}{}                                       & \multicolumn{1}{l}{}         \\
			\multicolumn{1}{l}{}                                       & \multicolumn{1}{l}{}        
		\end{tabular}
	\end{table}

\clearpage
\section{Triple points on $X_{2,4}\subset\mathbb{P}^5$}
We consider a threefold that is an intersection of quadric and quartic fourfolds in $\mathbb{P}^5$. We use $[x:y:z:t:u:w]$ as variables.
\begin{theorem}\label{6P}
	A complete intersection threefold $X_{2,4}\subset\mathbb{P}^5$ with only ordinary triple points as singularities can contain
	at most $10$ such points.
\end{theorem}
\begin{proof}
At an ordinary triple point the complete intersection threefold in $\mathbb{P}^5$ locally looks like the intersection of a hyperplane and a fourfould with an ordinary triple point. Let us assume the point in question in $O=[1:0:0:0:0:0]$. Let $G_i$ be a homogeneous polynomial of degree $i$ independent of the first variable $x$. Then it is easy to see that the complete intersection $X_{2,4}$ of a quadric $X_2$ and a quartic $X_4$ with respective equations $F_2:=xG_1+G_2$ and $F_4:=xG_3+G_4$ does indeed have a triple point at $O$. Note that $G_3$ needs to be a polynomial defining a smooth threefold in $\mathbb{P}^4$ for this singularity to be an ordinary one. In the next paragraph we show that for an ordinary triple point $O$ on $X_{2,4}$ we can always find $X_2$ and $X_4$ with equations of the above form.

We fix $X_2=V(F_2)=V(xG_1+G_2)$ as a quadric that defines $X_{2,4}$. Let $\hat{X}_4$ be a quartic fourfold that forms a regular sequence with $X_{2}$. Let $\hat{F}_4$ be the polynomial defining $\hat{X}_4$ and let us group the terms of this polynomial with respect to $x$, that is, $\hat{F}_4=\alpha x^4+x^3\hat{G}_1+\dots+\hat{G}_4$ with $\hat{G}_i$ a homogeneous polynomial independent of $x$ of degree $i$ and $\alpha$ a constant. We want $O$ to be a triple point, and so in particular $\hat{X}_4$ has to pass through $O$, which means $\alpha=0$. If $G_1$ defined a hyperplane different from $\hat{G}_1$ then the complete intersection of $X_2$ and $\hat{X}_4$ would be smooth at $O$ and so $G_1=\beta \hat{G}_1$ with $\beta$ being a constant. (It may happen in particular that $G_1$ is the zero polynomial but this does not cause any problems.) For the complete intersection threefold to have a triple point at $O$ we need $\hat{G}_2=\beta G_2+G_1H_1$ with $H_1$ being a degree one polynomial independent of $x$ and $\beta$ a constant. Otherwise, our intersection would have a double point at $O$. Consequently, we can write $\hat{F}_4=x^3\beta G_1+x^2(\beta G_2+G_1H_1)+x\hat{G}_3+\hat{G}_4$. We can always divide by a nonzero constant and assume $\beta=1$ (and abuse notation by leaving $H_1$ unchanged). In that case $F_4=\hat{F}_4-x^2F_2-xH_1F_2$ is the polynomial defining another fourfold $X_4$ containing $X_{2,4}$ and it is indeed of the desired form, thus proving the assertion from the previous paragraph.

From now on we assume that $X_{2,4}$ is defined by $F_2=xG_1+G_2$ and $F_4=xG_3+G_4$ as above. We can perform a projection from $O$ onto a $\mathbb{P}^4$ and obtain a quintic threefold. Indeed, let $E$ denote the exceptional divisor of a blowup of $O$. The projection is given by a linear system $|H-E|$, that is a system of strict transforms of hyperplanes passing through $O$. Since $(H-E)^3=5$, as a result of this projection we obtain a threefold of degree $5$ which we denote $X_5$. We want to show that $X_5$ is normal. Assume to the contrary that it is not. Then by \cite[p. 254]{JR} there would exist a normalization $\bar{X}_5$ with a two-dimensional ramification locus $D$ on $X_5$. Let $P$ be a point on the threefold $X_5$ and let $L_P$ be the line in $\mathbb{P}^5$ passing through $O$ that projects to $P$. Note that $P\in D$ if $L_p$ cuts $X_{2,4}$ either in two points other that $O$ or if it is tangent to $X_{2,4}$ at some point $\neq O$. We see that lines cutting $X_{2,4}$ twice or tangent to it need to be contained both in $X_2$ and $X_4$. Indeed, counting the number of intersection points (with multiplicities in case of tangency), these lines cut the fourfolds in $3$ points (in case of $X_2$) or at a triple point and $2$ additional points (in case of $X_4$). Thus they are contained in $X_{2,4}$, which means that the projection from $O$ contracts a three-dimensional cone ruled over $D$, a contradiction. 

Every triple point $O_i$ other than $O$ on $X$ yields a triple point $O'_i$ on $X_5$, thus if $\mu_3(X)=n$ then $\mu_3(X_5)=n-1$. Also note, that by the choice of equations for $X_2$ and $X_4$ we can consider $G_1G_4-G_2G_3$ to be the polynomial defining $X_5$ in $\mathbb{P}^4$ and thus obtain $24$ double points on $X_5$, namely those for which $G_1=G_2=G_3=G_4=0$. By an argument similar to \cite[2.6]{KR} we can consider Varchenko's spectral bound for $Q$. For $\alpha=2/5$ the spectrum of the fivefold point has length $155$  in the interval $(2/5, 7/5)$ whilst the spectrum of the triple point has length $14$ and the spectrum of the double point has length $1$. Thus from $155-24=131$ and $10\cdot14=140$ a quintic threefold containing $24$ double points can have at most $9$ ordinary triple points as the remaining singularities. As then $n-1=9$ we conclude that $X_{2,4}$ can contain at most $n=10$ ordinary triple points.
\end{proof}
\begin{lemma}
	There exitsts a threefold $X_{2,4}$ containing $7$ ordinary triple points as the only singularities.
	\end{lemma}
\begin{proof}
We can easily produce a fourfold $X_4$ admitting $6$ ordinary triple points in general linear position. Indeed the space of quartic fourfolds is of dimension $125$ and a triple point imposes $20$ linear conditions. Let $X_4$ be a quadric fourfold having triple points $[1:1:1:1:1:1]$ and $[1:0:0:0:0:0],\dots, [0:0:0:0:1:0]$. Let $\tilde{F}_4(x,y,z,t,u,w)$ be the polynomial defining $X_4$. We can write $\tilde{F}_4=w^3G_1+w^2G_2+wG_3+G_4$ where $G_i$ are homogeneous degree $i$ polynomials independent of $w$ such that $V(G_1)$ and $V(G_2)$ pass through the points $[1:1:1:1:1], [1:0:0:0:0],\dots, [0:0:0:0:1]$ in $\mathbb{P}^4$. Now lets take $X_2=V(\tilde{F}_2)$ passing through these points with $\tilde{F}_2=wG_1+G_2$. For $G_i$ general enough the complete intersection $X_{2,4}$ will have an ordinary triple point at $[0:0:0:0:0:1]$ and inherit all the ordinary triple points of $X_4$ producing an $X_{2,4}$ with $7$ OTPs.
We are unaware of any $X_{2,4}$ having more than $7$ OTPs as the only singularities.
\end{proof}

\section{Triple points on $X_{3,3}\subset\mathbb{P}^5$}
We consider a degree $9$ threefold that is an intersection of two degree $3$ hypersurfaces in $\mathbb{P}^5$. As before, we assume that $X_{3,3}$ has ordinary triple points as the only singularities. Let $X_{3,3}$ be $V(F_3, G_3)$ where $F_3$ and $G_3$ are the degree $3$ polynomials defining fourfolds in $\mathbb{P}^5$. Any other cubic fourfold containing $X_{3,3}$ is an element of the pencil $\alpha F_3+\beta G_3$ for $\alpha,\beta \in\mathbb{C}$. We denote this pencil by $\mathscr{L}$ and the fourfold defined by the above equation by $X_{\alpha, \beta}$. 

In our discussion we distinguish between the triple points that we call inherited and noninherited. For the former, we assume that in the pencil of cubic fourfolds there exists a fourfold $X_O$ for which the point $O$ is triple and thus the triple point on $X_{3,3}$ can be obtained by intersecting this singular fourfold with a smooth one. The latter case is when there is no fourfold in the pencil having an (ordinary) triple point at $O$ and thus the triple point on $X_{3,3}$ is specifically obtained because of the intersection of fourfolds, and not inherited from one of them. We only analyse the first case. 

\begin{theorem}
Let $X_{3,3}$ be a complete intersection threefold in $\mathbb{P}^5$ having only ordinary triple points as singularities. Furthermore, assume that for every triple point $O_i$ on $X_{3,3}$ there exists a cubic fourfold $X_i$ containing $X_{3,3}$ such that $O_i$ is triple on $X_i$. Then $X_{3,3}$ can have at most $9$ ordinary triple points. 
\end{theorem}

\begin{proof}
We begin by the analysis of points in general linear position. First, notice that a cubic fourfold having a triple point is necessarily a cone. Let us consider two triple points on $X_{3,3}$ denoted by $O_1$, $O_2$ and assume thery are inherited from $X_1$ and $X_2$ respectively. We observe that the line $L_{1,2}$ passing through $O_1$ and $O_2$ is contained in $X_{3,3}$ as it has at least a degree $4$ intersection with $X_1$ and $X_2$. On the other hand, it may happen that both points $O_1$ and $O_2$ are triple on the same fourfold $X_3$, which makes the line $L_{1,2}$ triple on $X_{3}$. Hence, there exists a point $O_3$ lying on $L_{1,2}$ which is triple on $X_{3,3}$ as the other cubic fourfold yielding the complete intersection threefold has to cut $L_{1,2}$ in three points in total. In short, we see that any two triple points on $X_{3,3}$ span a triple line contained in $X_{3,3}$ unless there is a third triple point of $X_{3,3}$ collinear with them. Analogously, we find that $3$ triple points in general linear position span a plane contained in $X_{3,3}$ and $4$ triple points in general linear position would span a $\mathbb{P}^3$ contained in $X_{3,3}$, which is a contradiction as this makes $X_{3,3}$ reducible. Thus there can be at most $3$ ordinary triple points in general linear position on $X_{3,3}$.

Let us see what happens when the points are not in general linear position. Assume there are $4$ OTPs on a line. Then the line is triple (by assumption) on one of the cubic fourfolds and contained in both (by degree count) and so singular on $X_{3,3}$, a contradiction. We show that there cannot be more than $8$ OTPs on the plane (the number $8$ is sufficient for our calculations yet it seems absurdly high and we believe $4$ triple points on a plane is the maximum). Remember our assumption that all the triple points are inherited from cubic fourfolds and thus the third triple point $O_3$ causes a plane $\Pi$ spanned by $O_1, O_2$ and $O_3$ to be contained in the intersection $X_{3,3}$. Assume that $\Pi=V(t,u,w)$. Since the plane has to be contained in every cubic fourfold $X_{\alpha,\beta}$ in the pencil $\mathscr{L}$ we can write the equation of the cubic in $\mathbb{P}^5\times\mathbb{P}^1$ as $X_{\alpha,\beta}=\alpha(tT_2+uU_2+wW_2)+\beta(tT'+uU'+wW')$ with $T,\dots ,  W'$ homogeneous in degree $2$. As every triple point $P$ has to be inherited from some fourfold, for each such point we need to be able to find parameters $\alpha_P$, $\beta_P$ such that the first and second derivatives of the equation $F$ defining $X_{\alpha_P,\beta_P}$ vanish. Moreover, we can restrict our analysis to the product of plane $\Pi$ with $\mathbb{P}^1$ since we assume that all the triple points lie on this plane. For brevity we use $\alpha_{i_m}$ to denote $\alpha$ or $\beta$, $x_{j_m}$ to denote variables from the set $\{x, y, z\}$, and $t_{k_m}$ to denote variables from the set $\{t, u, w\}$ for $m\in\{1,2\}$. Note that $\alpha_{i_1}$ may not be different from $\alpha_{i_2}$ etc. After restriction to $\Pi$ we have:
$$\frac{\partial ^2{F}}{\partial\alpha_{i_1}\partial\alpha_{i_2}}=\frac{\partial ^2{F}}{\partial x_{j_1}\partial x_{j_2}}=\frac{\partial ^2{F}}{\partial \alpha_{i_1}\partial x_{j_1}}=0, $$
$$\frac{\partial ^2{F}}{\partial{t_{k_1}}\partial{t_{k_2}}}=\alpha(\frac{\partial{T_{k_1}}}{\partial {t_{k_2}}}+\frac{\partial{T_{k_2}}}{\partial {t_{k_1}}})+\beta(\frac{\partial{T'_{k_1}}}{\partial {t_{k_2}}}+\frac{\partial{T'_{k_2}}}{\partial {t_{k_1}}}),$$
$$\frac{\partial ^2{F}}{\partial{\alpha_{i_1}}\partial{t_{k_1}}}=T_{k_1},$$
$$\frac{\partial ^2{F}}{\partial{x_{j_1}}\partial{t_{k_1}}}=\alpha{\frac{\partial T_{k_1}}{\partial x_{j_1}}}+\beta{\frac{\partial T'_{k_1}}{\partial x_{j_1}}}.$$
In this way we obtain $21$ degree $2$ equations that need to vanish simultaneously on $\Pi\times\mathbb{P}^1$. As three of them should cut out a finite set of points we see that a maximum number of triple points we can obtain that way is $8$. For reasons stated above any point outside the plane containing $3$ or more noncollinear triple points would cause $X_{3,3}$ to be reducible, a contradiction.

Assume now that in the pencil $\mathscr{L}$ there are two fourfolds $X_1$ and $X_2$ containing triple lines $L_1$ and $L_2$ respectively such that these lines do not meet. Then $X_1$ cuts $L_2$ in three points and similarly for $X_2$ and $L_1$. Each of these intersection points yields a triple point on $X_{3,3}$ and so we obtain $6$ triple points on $X_{3,3}$. As $L_1$ and $L_2$ are in general linear position we can assume $L_1=V(z,t,u,w)$ and $L_2=V(x,y,z,t)$. This means that $X_1=V(F(z,t,u,w))$ and $X_2=V(G_3(x,y,z,t))$, with each of the defining equations being independent of two out of six variables. Then the elements of the pencil $\mathscr{L}$ are as follows: $\bar{X}_{a,b}=V(aF_3+bG_3)$.
	Note that for $F(z,t,u,w)=a_1z^3+a_2t^3+a_3u^3+a_4w^3$, $G_3(x,y,z,t)=b_1x^3+b_2y^3+b_3z^3+b_4t^3$ with $a=b=1$ and $a_1=-b_3$ $a_2=-b_4$ we obtain a fourfold $X_{1,1}$ containing a triple line $L_3=V(x,y,u,w)$ providing another three triple points on $\bar{X}_{3,3}$ yielding a threefold $X_{3,3}$ with $9$ ordinary triple points as the only singularities. We cannot hope to obtain a $10$th triple point this way as we would always be able to find a plane $\Pi$ passing through one triple point of each of the lines $L_1, L_2, L_3$ and not containing the putative $P_{10}$ and thus giving $\mathbb{P}^3$ contained in $X_{3,3}$, a contradiction.
\end{proof}
	In particular, we have shown that $\mu_3(X_{3,3})=9$ if we assume that all the triple points are inherited, and $\mu_3(X_{3,3})\geq 9$ in general.

In the case when at least one triple point of $X_{3,3}$ is inherited from the triple point of some fourfold we have a nice characterization of $X_{3,3}$ as a double cover of a cubic threefold.

\begin{lemma}
	Assume that in the pencil of cubic fourfolds containing a complete intersection threefold $X_{3,3}\subset\mathbb{P}^5$ with only triple points as singularities there is a fourfold with a triple point. Then a threefold $\tilde{X}_{3,3}$ obtained by blowing up one of the triple points of $X_{3,3}$ is a double cover of a cubic threefold $Y_3$ ramified over a degree $12$ surface.
\end{lemma}
\begin{proof}
	Without loss of generality we may assume that $X_{3,3}$ is the complete intersection of a fourfold $X=V(x^2G_1+xG_2+G_3)$ and a cone $\bar{Y}=V(F_3)$ where $G_i$ and $F_i$ are homogeneous polynomials of degree $i$ independent of $x$. Note that in this case $P$ is exactly the vertex of the cone $\bar{Y}$ through which $X$ passes smoothly. A general line in $\mathbb{P}^5$ cuts $X$ in three points. As $X$ passes through $P$, this means that it cuts a general line of the ruling of the cone $\bar{Y}$ in $2$ other points. As we blow up $P$ and obtain $\tilde{X}_{3,3}$ we can treat it as a double cover of the base of the cone, that is $Y\subset\mathbb{P}^4$. We still have to check the ramification locus. This is exactly the vanishing locus of the discriminant, $D=V(G_2^2-4G_1G_3)$, calculated with respect to $x$ from the equation defining $X$. The intersection of $D$ and $Y$ in $\mathbb{P}^4$  is a degree $12$ surface $S_{12}$. This concludes the proof.
\end{proof}

\section{Triple points on $X_{2,2,3}$  in $\mathbb{P}^6$}
We provide an easily obtainable example of na $X_{2,2,3}$ threefold with $4$ ordinary triple points.
\begin{theorem}
There exists a complete intersection of two quadrics and a cubic in $\mathbb{P}^6$ with exactly $4$ ordinary triple points as the singular locus.
\end{theorem}
\begin{proof}
It is enough to choose a cubic fivefold having a triple plane and intersect it with two general quadric fourfolds. The complete intersection will have $4$ triple points lying on this plane as the only singularities.
\end{proof}
There is an interesting link between the geometry of the complete intersection threefolds $X_{3,3}$ and $X_{2,2,3}$, analogous to the one between $X_{2,4}$ and $X_5$ shown in the proof of the Theorem \ref{6P}.
\begin{lemma}
The projection from a triple point of a $X_{2,2,3}$ having $n$ triple points is a complete intersection $X_{3,3}$ threefold in $\mathbb{P}^5$ with at least $n-1$ triple points and $12$ double points.  
\end{lemma}
\begin{proof}
Let $F_3$ and $F_2=xA_1+A_2$, $G_2=xB_1+B+2$ denote homogeneous polynomials defining $X_{2,2,3}$ in $\mathbb{P}^5$. Note that $F_3$ is independent of $x$. Then $F_3=0$ and $A_2B_1-B_2A_1=0$ define the threefold $X_{3,3}$ that is the image of the projection from the triple point $O$ of $X_{2,2,3}$. Any triple points other than $O$ project to triple points on $X_{3,3}$ and we also additionally obtain $12$ double points whenever $A_1=A_2=B_1=B_2=F_3=0$, as the image of lines passing through $O$ contracted by this projection.
\end{proof}	

\section{Triple points on a sextic in $\mathbb{P}{[1:1:1:1:2]}$}
For background on weighted projective spaces (WPS) see for example \cite{ID} or \cite{MiR2}. 
\begin{theorem}\label{t6}
A sextic hypersurface $X_6$ in $\mathbb{P}[1:1:1:1:2]$ with only ordinary triple points as singularities can have at most $10$ ordinary triple points and this bound is attainable.
\end{theorem} 

We work in the weighted projective space $\mathbb{P}[1:1:1:1:2]$ with variables $x,y,z,t$ of weight $1$ and $u$ of weight $2$. We can thus write the equation of a sextic hypersurface as 
$F(x,y,z,t,u)=u^3+u^2G_2(x,y,z,t)+uG_4(x,y,z,t)+G_6(x,y,z,t)$ where $G_i$ is a homogeneous polynomial of degree $i$. Note that the term $u^3$ has to appear in this equation as otherwise the sextic would pass through the singular point of $WPS$ and inherit its singularity. 

\begin{lemma}
All triple points on a sextic hypersurface $X_6$ in $\mathbb{P}[1:1:1:1:2]$ have to lie on a sextic surface that is a hyperplane section of $X$.
\end{lemma}
\begin{proof}
For $O$ to be an (ordinary) triple point of $X_6$ we need all the second derivatives of $F$ to vanish. In particular, we need $\frac{\partial^2 F}{\partial u^2}=6u+2G_2(x,y,z,t)=0$ and thus $O$ has to lie on $V(3u+G_2(x,y,z,t))$ which is isomorphic to $\mathbb{P}^3$. This $\mathbb{P}^3$ cuts $X_6$ in a sextic surface.  
\end{proof}
We can perform a change of variables so the hyperplane section of $X_6$ is $u=0$. Thus we can consider triple points of $X_6$ as lying on a sextic surface in $\mathbb{P}^3$ which is the vanishing locus of $G_6(x,y,z,t)$. 
\begin{fact}
Triple points of a sextic threefold $X_6$ are also triple points of its hyperplane section $S_6=V(G_6)$. 
\end{fact}

From Proposition 3.2 in \cite{EPS} we know that if $S_6$ is a normal surface it can have at most $10$ (ordinary) triple points. Thus if we want to obtain more than $10$ triple points on $X_6$ we need to assume that $S_6$ is not normal. Then $S_6$ must have at least one-dimensional singular locus. We first assume that $\dim(\text{Sing}S_6)=2$. This is for example the case when $S_6$ is a triple quadric.
\begin{lemma}
A sextic threefold $X_6$ in $\mathbb{P}(1:1:1:1:2)$ whose triple points lie on a hyperplane section with a two-dimensional singular locus admits a singular curve.
\end{lemma}
\begin{proof}
We have $X_6=V(u^3+u^2G_2+uG_4+G_6)$. Note that on the curve $C=V(u)\cap V(G_4)\cap\text{Sing}S_6\subset X_6$ all partials derivatives of the equation defining $X_6$ vanish. This finishes the proof.
\end{proof}
From now on we are only concerned with the case where $S_6$ is a non-normal surface with one-dimensional singular locus.

\begin{proof}[Proof of \ref{t6}]
Let the defining equaton of $X_6$ be $F(x,y,z,t,u)=u^3+u^2G_2(x,y,z,t)+uG_4(x,y,z,t)+G_6(x,y,z,t)$. Let $G_6$ be the equation of a non-normal sextic surface with at least $11$ triple points and denote by $\Sigma$ the set of those triple points. We modify the argument regarding the polar bound to encompass the case where $S_6$ is not normal. Let us consider $|S_5|$, the linear system of polar surfaces of $S_6$ that is, the system of quintic surfaces generated by the partial derivatives of $G_6$. The base locus of this system is the union of a singular curve of $S_6$ and a set $\Sigma$ taken with multiplicity two, thus the general member $X_5$ of $|S_5|$ will be a quintic containing the singular curve of $S_6$ and having double points at $\Sigma$. Note that the dimension of intersection of $S_6$ and a general polar surface is $1$ as expected, and contains the singular curve of $S_6$. Now, $S_6$, $X_5$ and $X_4=V(G_4)$ should intersect only in a finite number of points. As we know that all these surfaces contain points of $\Sigma$ as triple or double points, we can calculate the degree of this intersection to be at least $11\cdot 3\cdot 2\cdot 2=132$ which is more than expected $120$ and thus there has to be at least a common curve of intersection for all these surfaces. As this calculation is valid for all the polar surfaces of $S_6$, this curve has to be contained in the singular curve of $S_6$, but then, from calculating the partials, it would be singular on $X_6$, contradicting the assumption that the singular locus of $X_6$ has only isolated ordinary triple points. 

To prove the final statement of the theorem it is enough to use the equation $X_6=V(u^3+G_6)$ where $G_6$ is any of the sextics with $10$ OTPs as in \cite{EPS} or \cite{JS}. It is clear that $X_6$ cannot be singular for $u\neq 0$, and for $u=0$ the only singularities are those on $V(G_6)=S$. Also, $X$ misses the singular point of the WPS as the monomial $u^3$ appears in its equation.
\end{proof}
\begin{remark}
In \cite[2,2]{SC} Cynk calculates the defect $\delta$ and the Hodge numbers of the resolution of singularities of an exemplary sextic threefold $X\subset\mathbb{P}(1:1:1:1:4)$ with equation of the form $G=u^3+G_6$ where $V(G_6)$ is one of the sextic surfaces with $10$ triple points as described in \cite{EPS}.  In particular, he obtains $\delta=10$, $h^{1,1}(\tilde{X})=21$ and $h^{1,2}(\tilde{X})=3$ where $\tilde{X}$ is the resolution of singularities of $X$.
\end{remark}
\section{Calabi-Yau threefolds with no triple points}
We briefly discuss the complete intersection threefold $X_{2,2,2,2}\subset\mathbb{P}^7$. In a similar spirit we mention two hypersurfaces in weighted projective space that cannot admit triple points.
\begin{proposition}
A complete intersection threefold $X_{2,2,2,2}\subset\mathbb{P}^7$ cannot have triple points.
\end{proposition}
\begin{proof}
We cannot obtain an ordinary triple point $O$ because the tangent cone of $X_{2,2,2,2}$ at $O$ would have to be a cone over a smooth cubic surface. This means that there could exist a change of variables such that locally around $O$ the equation of $X_{2,2,2,2}$ would be $F_3=x=y=z=0$. As all the sixfolds are of degree $2$, this cannot happen.
\end{proof}
\begin{proposition}
There can be no ordinary triple points on $X_8\subset\mathbb{P}{(1:1:1:1:4)}$ or $X_{10}\subset\mathbb{P}{(1:1:1:2:5)}$.
\end{proposition}
\begin{proof}
Both $X_8$ and $X_{10}$ can be considered a double cover of $\mathbb{P}^3$ and $\mathbb{P}(1:1:1:2)$ via a certain Veronese embedding. The ambient weighted projective space can be seen as a cone with the hyperplane section being isomorphic to the $\mathbb{P}^3$ (respectively $\mathbb{P}(1:1:1:2)$) and $X_8$ (resp. $X_{10}$) being the double section. Any line passing through a triple point of a cover would have to be contained in $X_8$ (resp. $X_{10}$). As this line passes through the vertex of the ambient cone, that means that $X_8$ (resp. $X_{10}$) would have to pass through this point as well, providing an additional, much more complicated singularity on $X_8$ (resp. $X_{10}$). 

This can also be seen algebraically, because the equation of $X_8\subset \mathbb{P}{(1:1:1:1:4)}$ (resp. $X_{10}\subset\mathbb{P}{(1:1:1:2:5)}$) necessarily involves a term with the last variable $u$ in the second power (otherwise the singular point of WPS would be contained in the threefold). As we can choose coordinates so that the putative triple point is $O=[1:0:0:0:0]$, we see that locally the equation of $X_8$ (resp. $X_{10}$) begins with $u^2$, meaning $O$ cannot be a triple point, a contradiction.
\end{proof}
\subsection*{Acknowledgements}
The author wishes to thank G. Kapustka, M. Kapustka and S. Cynk for helpful discussions. The author is also grateful to the reviewers for insightful and highly useful comments.
The author is supported by the project Narodowe Centrum Nauki 2018/30/E/ST1/00530.

\end{document}